\documentclass[11pt]{amsart}

\usepackage{geometry, amsmath, amssymb, amsthm, mathrsfs, setspace, comment, amscd, latexsym, mathtools, bm}
\usepackage[all]{xy}                          
\usepackage[colorlinks=true,
    citecolor=blue,
    linkcolor=blue,
    urlcolor=blue]{hyperref}

\usepackage{cleveref}

\usepackage{parskip}
\usepackage[abs]{overpic} 
\usepackage{pict2e}

\usepackage{graphicx, color}

\makeatletter
\@addtoreset{equation}{section}

\makeatother

\newtheorem{theorem}{Theorem}[section]
\newtheorem{lemma}[theorem]{Lemma}

\newtheorem{proposition}[theorem]{Proposition}

\theoremstyle{definition}

\newtheorem{remark}[theorem]{Remark}

\newcommand{\del}{\partial}

\newcommand{\Z}{\mathbb{Z}}
\newcommand{\R}{\mathbb{R}}
\newcommand{\Q}{\mathbb{Q}}
\newcommand{\C}{\mathbb{C}}
\newcommand{\CP}{\mathbb{CP}}
\newcommand{\M}{\mathcal{M}}

\newcommand{\U}{\mathcal{U}}

\renewcommand{\O}{\mathcal{O}}

\DeclareMathOperator{\can}{can}
\DeclareMathOperator{\st}{st}

\DeclareMathOperator{\ev}{ev}

\DeclareMathOperator{\FS}{FS}
\DeclareMathOperator{\Ker}{Ker}

\DeclareMathOperator{\PD}{PD}
\DeclareMathOperator{\Int}{Int}
\DeclareMathOperator{\Hom}{Hom}

\title[Symplectic fillings of unit cotangent bundles]{Symplectic fillings of unit cotangent bundles of spheres and applications}
\author[Myeonggi Kwon]{Myeonggi Kwon}
\address{Department of Mathematics Education, and Institute of Pure and Applied Mathematics, Jeonbuk National University, Jeonju 54896, Republic of Korea}
\email{mkwon@jbnu.ac.kr}

\author[Takahiro Oba]{Takahiro Oba}
\address{Department of Mathematics, The University of Osaka, 1-1 Machikane-yama, Toyonaka, Osaka 560-0043, Japan}
\email{taka.oba@math.sci.osaka-u.ac.jp}

\date{\today}

\begin{document}

\maketitle

\begin{abstract}
We prove the uniqueness, up to diffeomorphism, of symplectically aspherical fillings of the unit cotangent bundle of odd-dimensional spheres. 
As applications, we first show the non-existence of exact symplectic cobordisms between some $5$-dimensional Brieskorn manifolds. 
We also determine the diffeomorphism types of closed symplectic $6$-manifolds with certain codimension $2$ symplectic submanifolds. 
\end{abstract}

\section{Introduction}\label{section: intro}

This paper is concerned with the topology of symplectic fillings of a given contact manifold. 
Based on our interest, throughout the paper, symplectic filling always means \textit{strong symplectic filling}. 
We recall that a strong symplectic filling of a (cooriented) contact manifold $(Y,\xi)$ is a compact connected symplectic manifold $(W,\omega)$ satisfying that $\del W=Y$, and there exists a Liouville vector field $X$ for $\omega$, defined near and pointing outward along $\del W$, which satisfies $\xi=\Ker(i_{X}\omega|_{T\del W})$ (as cooriented contact structure). 
We require $(W,\omega)$ to be \textit{symplectically aspherical}, i.e. $[\omega]|_{\pi_2(W)}=0$, which suffices to ensure minimality of symplectic fillings $(W,\omega)$. 

Gromov \cite{Gro}, Eliashberg \cite{Eli90} and McDuff \cite{McDRational} independently showed that the standard contact $3$-sphere $(S^3, \xi_{\mathrm{st}})$ has a unique symplectically aspherical filling up to diffeomorphism. 
The uniqueness result for $(S^{2n-1}, \xi_{\mathrm{st}})$, $n \geq 3$, was proven by Eliashberg--Floer--McDuff \cite{McDContact}.   
The standard contact sphere $(S^{2n-1}, \xi_{\mathrm{st}})$ serves as the simplest example of a \textit{subcritically} Stein fillable contact manifold. 
Barth--Geiges--Zehmisch \cite{BGZ} showed the uniqueness of diffeomorphism types of symplectically aspherical fillings of a simply connected subcritically Stein fillable contact manifold $(Y,\xi)$ of $\dim Y \geq 5$.

A primary example of a \textit{critically} Stein fillable contact manifold is the unit cotangent bundle $ST^*B$ of a closed Riemannian manifold $B$ equipped with the canonical contact structure $\xi_{\can}$. 
This contact manifold has a natural symplectic filling, namely the disk cotangent bundle $DT^*B$ with the canonical symplectic structure $\omega_{\can}$. 
The question is whether there exists another symplectic filling. 
In the case $\dim B=2$, the answer to this question is somewhat well understood. 
McDuff \cite{McDRational} showed that all of symplectically aspherical fillings of $(ST^*S^2, \xi_{\can})$ are diffeomorphic to $DT^*S^2$, later which was strengthened to uniqueness up to Stein homotopy by Hind \cite{Hindlens}.
For $(ST^*T^2, \xi_{\can})$, Stipsicz \cite{Sti} showed that any Stein filling is homeomorphic to $DT^*T^2$, and then Wendl showed that $(DT^*T^2, \omega_{\can})$ is the unique symplectically aspherical filling up to symplectic deformation. 
Even in the case where $B$ is an orientable higher-genus or non-orientable surface, topological restrictions of symplectic fillings of $(ST^*B, \xi_{\can})$ are studied: \cite{LMY}, \cite{SVHM}, \cite{LinOz} and \cite{GZ}. 
In contrast, for the case $\dim B \geq 3$, the topology of symplectic fillings of $(ST^*B, \xi_{\can})$ is less explored. 
For $B=T^n$, Bowden--Gironella--Moreno \cite{BGM} and independently Geiges--Kwon--Zehmisch \cite{GKZ} showed the uniqueness of symplectically aspherical fillings up to diffeomorphism. 
In addition, Kwon--Zehmisch \cite{KZ} provided a uniqueness result for the diffeomorphism types of symplectically aspherical fillings of $(ST^*S^{2d+1}, \xi_{\can})$ ($d \geq 1$), assuming the existence of a certain complex hypersurface. 

The main theorem in this paper is a uniqueness result for symplectically aspherical fillings of $(ST^*S^{2d+1}, \xi_{\can})$, 
which has also recently been obtained by Bimmermann--Stratmann--Zehmisch \cite{BSZ}. 

\begin{theorem}\label{thm: uniqueness}
For any integer $d \geq 1$, let $(W, \omega)$ be a symplectically aspherical filling of the unit cotangent bundle $(ST^*S^{2d+1}, \xi_{\can})$.
Then, $W$ is diffeomorphic to $DT^*S^{2d+1}$.  
\end{theorem}

This theorem will be proved in Section~\ref{section: topology of fillings} after we discuss moduli spaces of holomorphic spheres in Section~\ref{section: holomorphic curves}. 
The present proof corrects and simplifies the earlier argument in the previous version and is inspired by the techniques of \cite{KZ, BSZ}. 
Following the approach of \cite{KZ, BSZ}, we construct a \textit{symplectic cap} of $(ST^*S^{2d+1},\xi_{\can})$ using the Audin--Lalonde--Polterovich \cite{ALP} Lagrangian embedding of $S^{2d+1}$ to close off the boundary of a symplectically aspherical filling $(W,\omega)$ of $(ST^{*}S^{2d+1}, \xi_{\can})$. 
We then study moduli spaces of holomorphic spheres in the resulting capped manifold. 
Applying the so-called degree method (see e.g.\ \cite{McDContact, BGZ}) to these moduli spaces yields topological information about $(W,\omega)$, which determines its fundamental group and homology. 
Finally, combining these computations with a result of \cite{KZ} establishes Theorem~\ref{thm: uniqueness}. 
A notable feature of our argument is that the properness of an evaluation map from a moduli space can be shown more directly than in \cite{BSZ}; see Lemma~\ref{compactness} and Remark~\ref{rmk: BSZ}.

We would like to note that, in a forthcoming paper \cite{KO_pi_1}, the authors prove that symplectically aspherical fillings of $(ST^*S^n,\xi_{\can})$ are simply connected for all $n \geq 3$ by using a different method.

In Section \ref{section: applications}, we will give two applications of the main theorem. 
The first application is to exact symplectic cobordisms. 
For each integer $\ell \geq 1$, define
$$
	Y_\ell = \{ (z_1,z_2,z_3, z_4) \in \C^4 \mid z_1^{2\ell}+z_2^2+z_3^2+z_4^2=0\} \cap S^7, 
$$
which is the so-called Brieskorn manifold of the singularity $z_1^{2\ell}+z_2^2+z_3^2+z_4^2$. 
We equip it with the contact structure $\xi_{\ell}$ induced by the standard one on $S^7$. 
Note that $(Y_1, \xi_1)$ is contactomorphic to $(ST^*S^3, \xi_{\can})$. 
It is known that all of the smooth manifolds $Y_{\ell}$ are diffeomorphic to $ST^*S^3$. 
In contrast, all of the contact manifolds $(Y_\ell, \xi_{\ell})$ are mutually \textit{non-contactomorphic} by a result of Uebele \cite{Ueb}. 
Therefore, there is no obstruction to the existence of smooth cobordisms between manifolds $Y_\ell$, and, moreover, there is an exact symplectic cobordism from $(Y_\ell, \xi_\ell)$ to $(Y_{\ell'}, \xi_{\ell'})$ if $\ell \leq \ell'$. 
However, unlike smooth cobordisms, the latter fact does not automatically guarantee the existence of a reverse exact symplectic cobordism; in fact, Kwon--van Koert \cite{KvK} posed a question about this existence. 
We partially answer this question in Theorem \ref{thm: cobordism}: there is no exact symplectic cobordism from $(Y_{\ell}, \xi_{\ell})$ to $(Y_1, \xi_1)$ for any $\ell\geq 2$.
A similar argument proves the non-existence of exact symplectic cobordisms from the exotic contact $5$-spheres studied in \cite{Usti} to the standard $5$-sphere $(S^5, \xi_{\mathrm{st}})$; see Remark \ref{rmk: sphere case}.

We will present another application of Theorem \ref{thm: uniqueness} in Section \ref{section: 6-manifolds}.   
McDuff \cite{McDRational} determined the symplectomorphism types of closed and connected symplectic $4$-manifolds containing  symplectic spheres with non-negative self-intersection. 
This reveals that the existence of a certain type of symplectic submanifolds constrains the topology of ambient symplectic $4$-manifolds. 
We capture a similar phenomenon in dimension $6$. 
A codimension $2$ symplectic submanifold $\Sigma$ of a closed symplectic manifold $(M, \Omega)$ with $[\Omega] \in H^2(M;\Z)$ is called a \textit{symplectic hyperplane section} on $(M, \Omega)$ if the cohomology class $[\Omega]$ is Poincar\'e dual to $[\Sigma]$. 
As the main result of Section \ref{section: 6-manifolds}, we prove that if a closed symplectic $6$-manifold $(M,\Omega)$ with $[\Omega] \in H^2(M;\Z)$ contains $(\CP^1 \times \CP^1, \omega_{\FS}\oplus \omega_{\FS})$ as a symplectic hyperplane section, then $M$ is diffeomorphic to a complex $3$-dimensional quadric (see Theorem \ref{thm: classification}). 
Note that replacing $(\CP^1 \times \CP^1, \omega_{\FS}\oplus \omega_{\FS})$ by $(\CP^2, \omega_{\FS})$ yields a similar result for this symplectic $4$-manifold; see Remark \ref{rem: CP^2}.
We would like to mention that B\v{a}descu \cite{Ba} 
showed an analogous result for complex projective $3$-folds.


\section{Moduli spaces of holomorphic spheres}\label{section: holomorphic curves}

\subsection{Closed symplectic manifold containing a symplectic filling}\label{section: closed}

Let $(W, \omega)$ be a symplectically aspherical filling of $(ST^{*}S^{2d+1}, \xi_{\can})$ with $d \geq 1$. 
In this paper, the holomorphic spheres under consideration lie in a closed symplectic manifold which contains $(W,\omega)$ as a symplectic submanifold. 
Such a manifold can be constructed using Lagrangian embeddings of odd-dimensional spheres due to Audin--Lalonde--Polterovich \cite{ALP}.


Consider the hyperplane bundle $\O_{\CP^d}(1)\to \CP^d$. 
Its first Chern class is represented by $\omega_{\FS}$, where $\omega_{\FS}$ denotes the normalized Fubini--Study form on $\CP^d$, defined so that
\begin{align*}
  \int_{\CP^1}\omega_{\FS}=1.
\end{align*}
There exists a disk bundle $N(\CP^d)\subset \O_{\CP^d}(1)$ equipped with a symplectic form $\Omega_0$ such that it restricts to $\omega_{\FS}$ on the zero-section and admits a Liouville vector field, defined away from the zero-section, which is tangent to each fiber and points inward along the boundary; in particular, the boundary $\partial N(\CP^d)$ is concave. 
Reversing the orientation of the fibers, $\partial N(\CP^d)$ becomes the Boothby--Wang bundle associated with the symplectic manifold $(\CP^d,\omega_{\FS})$; in fact, it is contactomorphic to the standard sphere $(S^{2d+1},\xi_{\st})$. 

Consider the closed disk $D^{2d+2}(1/\sqrt{\pi})\subset \C^{d+1}$ of radius $1/\sqrt{\pi}$, equipped with the standard symplectic form $\omega_{\st}$. 
Using Liouville flows near $\partial D^{2d+2}(1/\sqrt{\pi})$ and $\partial N(\CP^d)$, we can glue $D^{2d+2}(1/\sqrt{\pi})$ and $N(\CP^d)$ symplectically to obtain the closed symplectic manifold 
\[
  (M_0 = D^{2d+2}(1/\sqrt{\pi}) \cup N(\CP^d), \, \omega_0).
\]
(See \cite[Section 2.3]{KO} for further details on the construction of such glued symplectic manifolds.)

Now as in \cite{KZ} and \cite{BSZ}, consider the Lagrangian embedding of the unit sphere $S^{2d+1}=\{ z \in \C^{d+1} \mid \| z\| =1\}$ into $(D^{2d+2}(1/\sqrt{\pi}) \times \CP^d, \omega_{\st} \oplus \omega_{\FS})$, introduced by \cite{ALP}, which is defined by 
\begin{equation}\label{Lagr}
	\varphi \colon S^{2d+1} \to D^{2d+2}(1/\sqrt{\pi}) \times \CP^d, \quad \varphi(z)=(z/\sqrt{\pi}, (\overline{z})), 
\end{equation}
where  $\omega_{\st} \oplus \omega_{\FS}$ denotes the symplectic form $\mathrm{pr}_1^*\omega_{\st}+\mathrm{pr}_2^*\omega_{\FS}$ on $D^{2d+2}(1/\sqrt{\pi}) \times \CP^d$, where $\mathrm{pr}_i$ is the projection to the $i$-th factor. 
Since the inclusion $\iota \colon D^{2d+2}(1/\sqrt{\pi}) \times \CP^d \hookrightarrow M_0 \times \CP^d$ satisfies $\iota^{*} (\omega_0 \oplus \omega_{\FS}) = \omega_{\st} \oplus \omega_{\FS}$,
the image $L \coloneqq \varphi(S^{2d+1})$ is still Lagrangian in $(M_0 \times \CP^d, \omega_0 \oplus \omega_{\FS})$.  
Take a Weinstein tubular neighborhood $N(L)$ of the Lagrangian sphere $\varphi(S^{2d+1})$ in $M_0 \times \CP^d$ and replace it by $(W, \omega)$ symplectically. 
We denote the resulting glued symplectic manifold by 
\[ 
	(X=((M_0 \times \CP^d) \setminus N(L)) \cup W, \Omega). 
\]
Letting $\Sigma$ be the image of the zero-section of $\O_{\CP^d}(1) \to \CP^d$, 
we may assume that $X$ contains $\Sigma \times \CP^d \subset M_0 \times \CP^d$ by choosing $N(\varphi(S^{2d+1}))$ sufficiently thin.

\subsection{The moduli spaces}

To define moduli spaces of holomorphic spheres, we take an $\Omega$-tame almost complex structure $J$ on $X$ which agrees with the integrable almost complex structure $J_0$ on $N(\CP^d) \times \CP^d$ inherited from $\O_{\CP^d}(1) \times \CP^d$. 
Take a point $p_0 \in \Sigma$. 
Write $C$ for the second homology class of $X$ represented by a generic complex line in $\Sigma$. 

Now consider the moduli space $\widetilde{\M}$ consisting of all (parametrized) $J$-holomorphic spheres 
$u \colon \CP^1 \rightarrow X$ such that $[u]\coloneqq u_*([\CP^1])=C \in H_{2}(X;\Z)$ and 
$u(0) \in \{p_0\} \times \CP^d$: 
$$
	\widetilde{\M} = \left\{ 
	\begin{gathered}
	u \in W^{1,p}(\CP^{1}, X) \mid u \textrm{ is $J$-holomorphic, }[u]=C  \textrm{ and }u(0) \in \{p_0\} \times \CP^d 
	\end{gathered}
	\right\}.
$$
The group $G=\mathrm{Aut} (\CP^1; 0)$ of automorphisms of $\CP^1$ fixing $0$ freely acts on $\widetilde{\M}$ by 
$$
	\phi \cdot u \coloneqq u \circ \phi^{-1}
$$
for $\phi \in G$ and $u  \in \widetilde{\M}$; similarly, $G$ also acts on $\widetilde{\M} \times \CP^1$ by 
$$
	\phi \cdot (u, z) \coloneqq (u \circ \phi^{-1},  \phi(z)).
$$ 
Note that for $(u,z) \in \M \times \CP^1$, the marked point $z$ may agree with $0$. 
Set  
$$
	\M =\widetilde{\M}/G \quad \mathrm{and} \quad \U = (\widetilde{\M} \times \CP^1)/G. 
$$
We write $u$ for its equivalence class in $\M$ whenever no ambiguity arises. 

\begin{lemma}
The moduli spaces $\M$ and $\U$ are smooth manifolds of dimension $4d$ and $4d+2$, respectively, for a generic $J$. 
\end{lemma}

\begin{proof}
Since the $G$-actions on $\widetilde{\M}$ and $\widetilde{\M} \times \CP^1$ are free and proper (see e.g.~\cite[Lemma~3.1]{Kess}), the lemma follows immediately from the quotient manifold theorem once we verify that $\widetilde{\M}$ is a smooth manifold of dimension $4d+4$. 
Note that $\dim G = 4$.

Let us first compute the expected dimension of $\widetilde{\M}$. 
Over the image of $\CP^1=\CP^1 \times \{ \mathrm{pt}\}$ under the inclusion $\CP^1 \times \{ \mathrm{pt}\}  \hookrightarrow \Sigma \times \CP^d \hookrightarrow X$, the tangent bundle $TX$ splits into  
\begin{align}\label{eqn: splitting}
	TX|_{\CP^1} 
	\cong \ &(N_{(\Sigma \times \CP^d) /X} \oplus T(\Sigma \times \CP^d))|_{\CP^1} \cong 
	\O_{\CP^1}(1) \oplus T\Sigma|_{\CP^1} \nonumber  \\
	\cong \ &  \O_{\CP^1}(1) \oplus (\O_{\CP^1}(1) \oplus T\CP^{d-1}|_{\CP^1})  \nonumber\\
	\cong \ & \O_{\CP^1}(1) \oplus (\underbrace{\O_{\CP^1}(1) \oplus \cdots \oplus \O_{\CP^1}(1)}_{d-1}\oplus \O_{\CP^1}(2))\\
	 \cong \ & \O_{\CP^1}(1)^{\oplus d} \oplus \O_{\CP^1}(2).\nonumber
\end{align}
We obtain the same splitting over the sphere $\{ \mathrm{pt}\} \times \CP^1$. 
Thus, 
\begin{align*}
	c_1(TX)([u])=d c_1(\O_{\CP^1}(1))(C)+c_1(\O_{\CP^1}(2))(C)=d+2, 
\end{align*}
and the expected dimension of $\widetilde{\M}$ is given by 
\begin{align*}
	2c_1(TX)([u])+\frac{\dim X}{2} \cdot \chi(\CP^1)-\mathrm{codim}(\{p_0 \} \times \CP^d)=4d+4. 
\end{align*} 
Here, $\chi(\CP^1)$ denotes the Euler characteristic of $\CP^1$ and $\mathrm{codim}(\{p_0 \} \times \CP^d)$ denotes the codimension of the $2d$-dimensional submanifold $\{p_0 \} \times \CP^d$ in $X$. 

Since the algebraic intersection number between $C$ and $\Sigma \times \CP^d$ is $1$, 
positivity of intersections shows any holomorphic sphere belonging to $X$ is simply-covered. 
In light of \cite[Lemma 3.3.1]{MSBook}, the splitting (\ref{eqn: splitting}) implies that every element $u$ of $\widetilde{\M}$ whose image is contained in a neighborhood of $\Sigma \times \CP^d$ is Fredholm regular. 
Moreover, \cite[Theorem 4.4.3]{Wen_Lecture} guarantees the existence of a generic extension $J$ of $J_0$ to make the other $J$-holomorphic spheres all Fredholm regular with respect to $J$. 
Therefore, for such a generic $J$, \cite[Theorem 3.4.1]{MSBook} concludes that $\widetilde{\M}$ is a smooth manifold of dimension $4d+4$. 
\end{proof}

Consider the following commutative diagram of smooth manifolds:
\[
\begin{xy}
  \xymatrix{
    \widetilde{\M} \times \CP^1 \ar[d]_{\mathrm{pr}_1} \ar[r]^{\hspace{18pt}/G} &  \mathcal{U} \ar[d]^{\mathfrak{f}} \\
    \widetilde{\M} \ar[r]_-{/G} & \M
  }
\end{xy}
\]
where the horizontal arrows are the quotient maps by the $G$-actions, $\mathrm{pr}_1$ denotes the first projection and 
\begin{align*}
	\mathfrak{f} \colon \U \longrightarrow \M.
\end{align*}
is the map induced by $\mathrm{pr}_1$. 
In view of the freeness and properness of the $G$-actions, it follows from a standard argument in the theory of smooth manifolds that the map $\mathfrak{f}$ is smooth and, in fact, a $\CP^1$-bundle.

In other words, $\U$ serves as the \textit{universal family} over $\M$. 
We now point out that $\mathfrak{f}$ admits two distinguished sections, which will play an important role in the next section. 
Consider the evaluation map 
\begin{align*}
	\ev \colon \mathcal{U} \longrightarrow X, 
	\qquad \operatorname{ev}([u,z]) = u(z).
\end{align*}
The first section of the $\CP^1$-bundle is 
\[
	\operatorname{ev}^{-1}(\{p_0\} \times \mathbb{CP}^d).
\] 
For the second section, take a generic holomorphic section of 
$\mathcal{O}_{\mathbb{CP}^d}(1) \to \mathbb{CP}^d$ 
whose image lies in $N(\mathbb{CP}^d)$ and is disjoint from $p_0 \in \Sigma$. 
Denote the image of this section by $\Sigma'$. 
Then 
\[
	\operatorname{ev}^{-1}(\Sigma' \times \mathbb{CP}^d)
\]
defines another section. 
These are indeed sections, as follows from positivity of intersections together with the fact that the algebraic intersection numbers of $[\Sigma\times \CP^d]=[\Sigma'\times \CP^d]$ with $C$ are both equal to $1$. 
Moreover, since $p_0$ and $\Sigma'$ are disjoint, the two sections constructed above are disjoint as well.

\subsection{Compactness of the moduli spaces}

Here we show the compactness of the moduli spaces. 

\begin{lemma}\label{compactness}
The moduli spaces $\M$ and $\U$ are compact. 
In particular, the evaluation map $\ev \colon \U \to X$ is proper. 
\end{lemma}

\begin{proof}
It suffices to show the compactness of $\M$. 
The following argument is inspired by \cite[Section 3.4]{KZ}. 
Let $(u_\nu)$ be a sequence of holomorphic spheres $u_\nu \in \M$. 
By Gromov compactness, passing to a subsequence, $(u_{\nu_k})$ converges to a stable holomorphic sphere 
\begin{align*}
	u_\infty=(u_\infty^1, \ldots, u_\infty^m) \in \overline{\M},  
\end{align*}
where $\overline{\M}$ is the Gromov compactification of $\M$. 

We first prove that exactly one irreducible component among $u_\infty^{j}$ ($j=1,\ldots,m$) is non-constant. 
The proof proceeds by contradiction, using symplectic energy.
Assume that $m \geq 2$. 
We may suppose that all $u_\infty^j$ are non-constant. 
By the Mayer--Vietoris exact sequence together with the simply-connectedness of $ST^*S^{2d+1}=\del W$, 
the map 
\begin{align*}
	H_2(W; \Z) \oplus H_2((M_0 \times \CP^d) \setminus N(L); \Z) \to H_2(X; \Z)
\end{align*}
is surjective. 
Since $(W,\omega)$ is symplectically aspherical, the elements of $H_2(W;\Z)$ have zero symplectic energy and hence contribute nothing to the energy. 
Similarly, one finds that the map 
\begin{align*}
	H_2(N(L); \Z) \oplus H_2((M_0 \times \CP^d) \setminus N(L); \Z) \to H_2(M_0 \times \CP^d; \Z)
\end{align*}
is surjective. 
When $d > 1$, this map is in fact an isomorphism since $H_2(ST^{*}S^{2d+1};\Z)=0$. 
Thus, $H_2((M_0 \times \CP^d) \setminus N(L); \Z)$ is generated by $C$ and $D$, where $D$ is represented by a complex line in $\{ \mathrm{pt} \} \times \CP^d$. 
When $d=1$,  in addition to those generators, $H_2((M_0 \times \CP^d) \setminus N(L); \Z)$ has a fiber class of the $S^2$-bundle 
\begin{align*}
	ST^{*}S^3 \cong S^3 \times S^2 \to S^3
\end{align*}
as observed in \cite[Section 3.4.3]{KZ}. 
However, since the symplectic area of this class is $0$, it can be ignored. 
Therefore, we may express each homology class $[u_{\infty}^j]$ as 
\begin{align*}
	[u_{\infty}^j]=a_jC+b_jD+S_j, 
\end{align*} 
where  $a_j, b_j \in \Z$ and $S_j$ is the image of an element of $H_2(W;\Z)$.

By positivity of intersections, 
\begin{align*}
	[u_\infty^j] \cdot [\Sigma \times \CP^d]=a_j \geq 0. 
\end{align*}
Since $C \cdot [\Sigma \times \CP^d]=1$, it follows that exactly one $a_j$ equals $1$ while the others vanish. 
Without loss of generality, assume this occurs for $j=1$: 
\begin{align*}
	[u_\infty^1]=C+b_1D+S_1 \quad \textrm{ and } \quad [u_\infty^j]=b_jD+S_j \textrm{ for }j=2,\ldots, m. 
\end{align*}
As $[u_\infty]=[u_\nu]=C$, we obtain 
\begin{align*}
	\Omega([u_{\infty}])  =  \Omega(C)=1. 
\end{align*}
On the other hand, $[u_\infty]= [u_\infty^1]+\cdots+[u_\infty^m]=C+(\sum_{j=1}^{m}b_j)D+\sum_{j=1}^{m}S_j$, so 
\begin{align*}
	\Omega([u_{\infty}])  =  \Omega\left(C+\left(\sum_{j=1}^{m}b_j \right)D+\sum_{j=1}^{m}S_j \right) =1+ \sum_{j=1}^mb_j. 
\end{align*}
Thus, $\sum_{j=1}^mb_j=0$, that is, 
\begin{align}\label{inq: d}
b_1= -\sum_{j=2}^mb_j.
\end{align} 
Since each $u_j^{\infty}$ is a non-constant holomorphic sphere, its symplectic energy is positive, which implies 
\begin{align*}
	b_1> -1 \quad \textrm{ and } \quad b_j>0 \textrm{ for }j=2,\ldots, m.
\end{align*}
Combining this with the inequality (\ref{inq: d}) yields 
\begin{align*}
	-1 < b_1 <0,
\end{align*}
contradicting the integrality of $b_1$. 
Hence $u_{\infty}$ has a unique non-constant component. 

Finally we check that the limit $u_\infty=(u_\infty^1, \ldots, u_\infty^m)$ has no constant component, and therefore $u_\infty \in \M$. 
Suppose on the contrary that a constant component exists. 
Since $u_\infty$ has only one non-constant component, this forces the existence of a constant component with at least two marked points besides nodal points. 
This is impossible by the definition of the moduli space $\M$, which completes the proof.
\end{proof}

\begin{remark}\label{rmk: BSZ}
In \cite{KZ} and \cite{BSZ}, the ambient space $\CP^1 \times \C^d \times \CP^d$ is obtained by the one-point compactification of the first $\C$-factor of $\C^{d+1} \times \CP^d$, while our $M_0 \times \CP^d$ is obtained by capping off the $\C^{d+1}$-factor of $\C^{d+1} \times \CP^d$ using a compactifying divisor. 
These make the difference between \cite{KZ, BSZ} and ours to show the properness of the evaluation map: 
in \cite{KZ, BSZ}, their proof crucially relies on the existence of two hypersurfaces with desired intersection behavior, one of which is demanding to find; 
on the other hand, our proof uses only one canonical hypersurface $\Sigma \times \CP^d$. 
\end{remark}

\begin{remark}\label{rem: connected}
The moduli space $\M$ a priori is not connected; neither is $\U$ as it is a $\CP^1$-bundle over $\M$. 
Still, all holomorphic spheres contained in $\Sigma \times \CP^d$ belong to the same connected component of $\M$. 
Hence, taking this connected component, we assume $\M$ and $\U$ to be connected throughout this paper. 
\end{remark}

\section{Topology of symplectic fillings}\label{section: topology of fillings}

\subsection{Degree of the evaluation map}

Let $(W,\omega)$ be a symplectically aspherical filling of $(ST^{*}S^{2d+1}, \xi_{\can})$ with $d \geq 1$ and $(X,\Omega)$ the closed symplectic manifold constructed as in Section \ref{section: closed}.  
We begin this section by analyzing the evaluation map $\ev \colon \U \rightarrow X$ defined to be $\ev([u,z])=u(z)$. 

\begin{lemma}\label{lem: unique}
Given a point $q_0 \in (\Sigma \setminus\{p_0\}) \times \CP^d$, there exists a unique holomorphic sphere $u \in \M$ passing through $q_0$. 
\end{lemma}

\begin{proof}
Since $C \cdot [\Sigma \times \CP^d]=1$, positivity of intersections shows that the image of an element $u$ of $\M$ passing through $q_0$ must entirely lie in $\Sigma \times \CP^d$. 
By the choice of the almost complex structure $J$, such a holomorphic sphere is a genuine complex curve on $\Sigma \times \CP^d$. 

We denote a curve in question by $u=(u_1, u_2)$ with complex curves $u_1 \colon \CP^1 \to \Sigma$ and $u_2 \colon \CP^1 \to \CP^d$. 
Since $[u]=C$, we have $[u_2]=0$, which implies that $u_2$ is constant. 
As for $u_1$, this is a complex line in $\Sigma \cong \CP^d$ passing through the two distinct points $p_0$ and $q_{01}$, where $q_{01}$ denotes the $\Sigma$-component of $q_0 \in \Sigma \times \CP^d$.
This concludes the uniqueness of holomorphic spheres in question.   
\end{proof}

\begin{proposition}\label{prop: degree}
The evaluation map $\ev \colon \U \rightarrow X$ has degree $1$. 
It restricts to a proper degree $1$ map $\ev \colon \U \setminus \ev^{-1}( (\Sigma' \times \CP^d))  \rightarrow X \setminus (\Sigma' \times \CP^d)$.
\end{proposition}

\begin{proof}
Take a point $q_0 \in (\Sigma \setminus \{p_0\}) \times \CP^d$ as in the preceding lemma.
By that lemma, the fiber $\ev^{-1}(q_0)$ consists of a single element, namely $\ev^{-1}(q_0)={[u,z]}$.
In view of the splitting of $u^*TX$ in (\ref{eqn: splitting}), it follows that $q_0$ is a regular value of $\ev$.
Hence $\deg(\ev)=1$.
The assertion concerning the restricted evaluation map then follows directly.
\end{proof}

\subsection{Uniqueness of symplectic fillings}

Here we prove Theorem \ref{thm: uniqueness}. 
The proof relies on the following lemma, which is essentially equivalent to a result of Kwon--Zehmisch \cite[Lemma 3.1]{KZ}.

\begin{lemma}[{Kwon--Zehmisch \cite[Lemma 3.1]{KZ}}]\label{lem: KZ}
If the embedding $\iota \colon \{ p_0 \} \times \CP^d \to X \setminus (\Sigma'\times \CP^d)$ is $\pi_1$-surjective and $H_k(\, \cdot\, ; \mathbb{F})$-onto for all $k \in \Z$ and all fields $\mathbb{F}$, then $W$ is diffeomorphic to $DT^*S^{2d+1}$. 
\end{lemma}

\begin{proof}[Proof of Theorem \ref{thm: uniqueness}]
Recall that $\ev^{-1}(\{p_0\}\times \CP^d)$ and $\ev^{-1}(\Sigma'\times \CP^d)$ give disjoint sections of 
$\mathfrak{f}\colon \mathcal{U}\to \M$. 
Moreover, since $\mathfrak{f}$ is a $\CP^1$-bundle, the space 
$\mathcal{U}\setminus \ev^{-1}(\Sigma'\times \CP^d)$ deformation retracts onto 
$\ev^{-1}(\{p_0\}\times \CP^d)$. 
Thus we obtain a commutative diagram
\[
\begin{xy}
  \xymatrix{
    \ev^{-1}(\{p_0\}\times \CP^d) \ar[d]_{\simeq} \ar[r]^{\hspace{14pt}\ev} &
    \{p_0\}\times \CP^d \ar[d]^{\iota} \\
    \mathcal{U}\setminus \ev^{-1}(\Sigma'\times \CP^d) \ar[r]_-{\ev} & 
    X \setminus (\Sigma'\times \CP^d)
  }
\end{xy}
\]
where the vertical arrows are inclusions. 
As $\ev \colon \mathcal{U}\setminus \ev^{-1}(\Sigma'\times \CP^d)\to X \setminus (\Sigma'\times \CP^d)$ has degree $1$ by Proposition \ref{prop: degree}, 
it induces a surjection on $\pi_1$ and is onto on $H_k(\,\cdot\,;\mathbb{F})$ for all $k\in \mathbb{Z}$ and all fields $\mathbb{F}$. 
In view of the homotopy equivalence 
$\ev^{-1}(\{p_0\}\times \CP^d) \simeq  
\mathcal{U}\setminus \ev^{-1}(\Sigma'\times \CP^d)$, 
the same holds for $\iota$. 
Therefore, Lemma \ref{lem: KZ} implies that $W$ is diffeomorphic to $DT^*S^{2d+1}$. 
\end{proof}

\begin{remark}
The above proof can be compared to the proof of the uniqueness of symplectically aspherical fillings of $(S^{2n-1}, \xi_{\st})$ given by Ghiggini--Niederkr\"uger \cite[Corollary~5.3]{GN}. 
In their proof, they consider the closed symplectic manifold $(X',\Omega')$ obtained from a given symplectically aspherical filling $(W',\omega')$ of $(S^{2n-1}, \xi_{\st})$ by capping with the disk bundle over $\CP^{n-1}$ associated with $\O_{\CP^{n-1}}(1)$. 
Then they determine the fundamental group and homology of $W'$ using the moduli space of holomorphic spheres $\M'$ in $X'$ passing through a fixed point of $\CP^{n-1} \subset X'$ and the universal family $\U'$ over $\M'$. 
Finally, they apply the h-cobordism theorem to determine the diffeomorphism type of $W'$.
We observe that the moduli spaces $\M'$ and $\U'$ are similar to our $\M$ and $\U$, respectively.
\end{remark}


\section{Applications}\label{section: applications}

\subsection{Symplectic cobordisms}\label{section: cobordism}

Let  $(Y_+, \xi_+=\Ker(\alpha_+))$ and $(Y_-, \xi_-=\Ker(\alpha_-))$ be (cooriented) contact manifolds. 
A compact exact symplectic manifold $(X, d\lambda)$ with boundary $\del X=-Y_- \sqcup Y_+$ is called an \textit{exact symplectic cobordism} from $(Y_-, \xi_-)$ to $(Y_+, \xi_+)$ if the Liouville vector field on $X$, defined to be $d\lambda$-dual to $\lambda$, points  inward along $Y_-$ and outward along $Y_+$, and $\xi_{\pm}=\Ker(\lambda|_{TY_{\pm}})$. 
Here $Y_{\pm}$ and $X$ are oriented by the volume forms $\alpha_{\pm} \wedge (d\alpha_{\pm})^{n-1}$ and $(d\lambda)^n$, respectively, and $-Y_-$ means $Y_-$ with the reversed orientation. 
If $(Y_-, \xi_-)$ is the empty set, then the exact symplectic cobordism $(X, d\lambda)$ is called an \textit{exact symplectic filling} of $(Y_+, \xi_+)$. 

Recall that for each integer $\ell \geq1$ the $5$-manifold $Y_\ell$ is defined to be 
$$
	Y_\ell = \{ (z_1,z_2,z_3, z_4) \in \C^4 \mid z_1^{2\ell}+z_2^2+z_3^2+z_4^2=0\} \cap S^7. 
$$
This is the so-called Brieskorn manifold, known to be diffeomorphic to $S^2 \times S^3$ (see e.g. \cite[Proposition 6.1]{DK}). 
It also carries the canonical contact structure $\xi_\ell$, induced by the standard contact structure on the unit $7$-sphere $S^7=\{z \in \C^4 \mid \|z\|=1\}$. 
The Brieskorn contact manifold $(Y_1, \xi_1)$ is contactomorphic to the unit cotangent bundle $(S^*TS^3, \xi_{\can})$ (see for example {\cite[Lemma 3.1]{KvK}} and \cite[Lemma 5.3]{Mayd}). 
A Milnor fiber $V_\ell$ of the singularity $z_1^{2\ell}+z_2^2+z_3^2+z_4^2$ gives an exact symplectic filling of $(Y_\ell, \xi_\ell)$. 
According to \cite[Section 2.7]{AGLV}, for $\ell$ and $\ell'$ with $\ell<\ell'$, the singularity $z_1^{2\ell}+z_2^2+z_3^2+z_4^2$ is adjacent to $z_1^{2\ell'}+z_2^2+z_3^2+z_4^2$; hence \cite[Lemma 9.9]{Keating} leads to an exact symplectic embedding of a Milnor fiber $V_{\ell}$ into the completion of $V_{\ell'}$. 
The latter is diffeomorphic to $V_{\ell'} \cup ([0,\infty) \times Y_{\ell'})$, and the image of the embedding is contained in $V_{\ell'} \cup ([0,c) \times Y_{\ell'})$ for some $c>0$ by compactness. 
Thus, truncating the completion at the hypersurface $\{c\} \times Y_{\ell'}$ and removing the interior of the embedded $V_{\ell}$ shows the existence of an exact symplectic cobordism from $(Y_\ell, \xi_\ell)$ to $(Y_{\ell'}, \xi_{\ell'})$. 
The following theorem tells us that this does not hold for the reverse direction in general. 

\begin{theorem}\label{thm: cobordism}
For any $\ell >1$, there exists no exact symplectic cobordism from $(Y_\ell, \xi_\ell)$ to $(Y_1, \xi_1)$. 
\end{theorem}

\begin{proof}
Suppose for the sake of contradiction that there exists an exact symplectic cobordism $(W_{\ell,1},d\lambda)$ from $(Y_\ell, \xi_\ell)$ to $(Y_1, \xi_1)$. 
Then, gluing $(W_{\ell,1}, d\lambda)$ with $V_\ell$ symplectically gives an exact symplectic filling $(W,\omega)$ of $(Y_1, \xi_1)$. 
Consider the Mayer--Vietoris exact sequence for the pair $(V_\ell, W_{\ell,1})$: 
\begin{align}\label{seq: cobordism}
	H_3(Y_\ell;\Q) \rightarrow H_3(V_\ell;\Q) \oplus H_3(W_{\ell,1}; \Q) \rightarrow H_3(W;\Q) \xrightarrow{\varphi} H_2(Y_\ell; \Q) 
\end{align}
We have $H_3(Y_\ell; \Q) \cong H_2(Y_\ell; \Q) \cong \Q$ since $Y_\ell$ is diffeomorphic to $S^2 \times S^3$, and by Theorem~\ref{thm: uniqueness}, $H_3(W;\Q) \cong \Q$. 
As the Milnor fiber $V_\ell$ is homotopy equivalent to the wedge sum of $2\ell-1$ $3$-spheres by \cite[Theorems~6.5 and~9.1]{Milnor}, we obtain $H_3(V_\ell; \Q) \cong \Q^{2\ell-1}$. 
Moreover, this homotopy type implies that $H_4(V_\ell, Y_\ell) \cong H^2(V_\ell)=0$, which concludes that the induced map $H_3(Y_\ell; \Q) \rightarrow H_3(V_\ell;\Q)$ by the inclusion is injective; so is the first map in (\ref{seq: cobordism}). 
To sum up, the sequence (\ref{seq: cobordism}) becomes 
\begin{align}\label{seq: cobordism2}
	\Q \hookrightarrow \Q^{2\ell-1} \oplus H_3(W_{\ell,1}; \Q) \rightarrow \Q \xrightarrow{\varphi} \Q
\end{align}

Now there are two possibilities of $\varphi$: either it is the zero-map or an isomorphism. 
In the former case, the sequence (\ref{seq: cobordism2}) becomes a split short exact sequence. 
Hence, we have 
$$
	\Q^{2\ell-1} \oplus H_3(W_{\ell,1};\Q) \cong \Q^2, 
$$
which contradicts our assumption about $\ell$. 
Similarly, in the latter case, we have 
$$
	\Q^{2\ell-1} \oplus H_3(W_{\ell,1};\Q) \cong \Q,
$$
again a contradiction. 
\end{proof}

\begin{remark}
As mentioned above, $Y_{\ell}$ is diffeomorphic to $S^2 \times S^3$ for every $\ell$. In particular, there are no topological obstructions to the existence of \emph{smooth} cobordisms between $Y_{\ell}$'s. In this sense, Theorem 4.1 highlights intriguing differences between symplectic and smooth cobordism categories. 
\end{remark}

\begin{remark}
The same result as Theorem~\ref{thm: cobordism} also holds for the $(4d+1)$-dimensional Brieskorn manifolds
\[
  Y_{\ell}^{4d+1}
  = \{(z_1, z_2, \ldots,z_{2d+2}) \in \C^{2d+2} \mid z_1^{2\ell}+z_2^2+\cdots+z_{2d+2}^2=0\} \cap S^{4d+3}
\]
for any $d \geq 2$. 
This follows from essentially the same argument as in the above proof, 
making use of the middle-dimensional homology of a hypothetical exact symplectic cobordism.
Note that $Y_1^{4d+1}$, with its canonical contact structure, is contactomorphic to $(ST^*S^{2d+1}, \xi_{\can})$, just as in the case $d=1$. 
Moreover, as shown in \cite[Proposition~1]{Ueb}, the diffeomorphism types of $Y_{\ell}^{4d+1}$ depend on $\ell$ modulo $4$; in particular, $Y_{\ell}^{4d+1}$ with $\ell \equiv 1 \pmod{4}$ is diffeomorphic to $ST^*S^{2d+1}$. 
Thus, in this case there are no topological obstructions to the existence of an exact symplectic cobordism.
\end{remark}

\begin{remark}
Here are some remarks from the perspective of Floer theory. 
\begin{enumerate}

\item The $5$-dimensional Brieskorn manifolds $Y_{\ell}$ have provided particularly interesting examples of exotic contact structures. In \cite[Corollary 4]{Ueb}, it is proved that $Y_{\ell}$'s are pairwise non-contactomorphic using positive (non-equivariant) symplectic homology, while they have the same positive equivariant symplectic homology \cite[Appendix A]{KvK}; see also \cite[Question 6.10]{KvK}.

\item As an alternative approach to Theorem 4.1, one might try to show that there are no exact symplectic embeddings from the Milnor fiber $V_{\ell}$ to the Milnor fiber $V_{\ell'}$ if $\ell'$ is smaller than $\ell$; this implies that there are no cobordisms from $Y_{\ell}$ to $Y_{\ell'}$. Unfortunately, most of embedding obstructions in the literature are only available for domains in $\R^{2n}$ or more generally for domains with vanishing symplectic homology, whereas Milnor fibers we are interested in have nonvanshing symplectic homology; see \cite[Theorem 6.3]{KvK}.  An obstruction applicable to domains with nonvanishing symplectic homology comes from the notion of dilation recently introduced in \cite{Zhou}. Direct computations using \cite[Proposition 5.14]{Zhou} and \cite[Section 5.3.1]{KvK} however show that the Milnor fiber $V_{\ell}$ admits the same order of dilations for each $\ell$, so this obstruction does not work in our situation. As for another potential obstruction, we refer the reader to \cite{GaSie}. 
\end{enumerate}
\end{remark}

\begin{remark}\label{rmk: sphere case}
The Brieskorn manifold $Y'_{\ell}$ given by 
$$
	Y'_\ell = \{ (z_1,z_2,z_3, z_4) \in \C^4 \mid z_1^{2\ell-1}+z_2^2+z_3^2+z_4^2=0\} \cap S^7,
$$
with $\ell \geq 1$, is known to be diffeomorphic to $S^5$ (see e.g. \cite{Brie}) and has the canonical contact structure $\xi'_{\ell}$. 
One can verify that $(Y'_1, \xi'_1)$ is contactomorphic to the standard sphere $(S^5, \xi_\mathrm{st})$. 
Indeed, the submanifolds of $S^7$ defined by 
$$
	Y'_1(t)= \{ (z_1,z_2,z_3, z_4) \in \C^4 \mid z_1+t(z_2^2+z_3^2+z_4^2)=0\} \cap S^7,
$$
with $t \in [0,1]$, are all diffeomorphic to $S^5$, and $Y'_1(0)$ and $Y'_1(1)$ equipped with the contact structures induced by $(S^7, \xi_\mathrm{st})$ agrees with $(S^5, \xi_\mathrm{st})$ and $(Y'_1, \xi'_1)$, respectively; 
hence, identifying each $Y'_1(t)$ with $S^5$, the Gray stability theorem shows that they are contactomorphic.
It is shown in \cite{Usti} that for $2\ell-1 = \pm 1$ (mod $8$), $Y'_{\ell}$'s are pairwise non-contactomorphic to each other.

We can show that there exists an exact symplectic cobordism from $(Y'_1, \xi'_1)$ to $(Y'_\ell, \xi'_\ell)$ for each $\ell>1$ but no such a cobordism in the reverse direction, using the same idea as the above proof coupled with the uniqueness result on the diffeomorphism types of symplectically aspherical fillings of $(S^5, \xi_\mathrm{st})$ by Eliashberg--Floer--McDuff \cite[Theorem 1.5]{McDContact}. 
Moreover, this also proves the non-existence of exact symplectic cobordisms from $(Y'_\ell, \xi'_\ell)$ to $(Y_1, \xi_1)$ or from $(Y_\ell, \xi_\ell)$ to $(Y'_1, \xi'_1)$ for $\ell>2$ although there are smooth cobordisms between them.
\end{remark}


\subsection{Topology of closed symplectic $6$-manifolds}\label{section: 6-manifolds}

Here we will provide another application of Theorem~\ref{thm: uniqueness} in dimension $6$. 

Given a closed symplectic $2n$-manifold $(M,\Omega)$, with $[\Omega] \in H^2(M;\Z)$, having a symplectic hyperplane section $\Sigma$, there is a Liouville vector field $V$, defined near the boundary of a symplectic tubular neighborhood $N(\Sigma)$ of $\Sigma$ in $M$, which points inward along the boundary. 
We equip this boundary $\del N(\Sigma)$ with the contact structure defined by $\alpha=(i_V \Omega)|_{T\del N(\Sigma)}$. 
Then, this contact manifold is contactomorphic to the Boothby--Wang bundle associated with $(\Sigma, \Omega_\Sigma)$, where $\Omega_\Sigma$ denotes the restriction of $\Omega$ to $T\Sigma$. 
By \cite[Lemma 2.2]{DL}, the complement $M \setminus \Int(N(\Sigma))$ with $\Omega|_{M \setminus \Int(N(\Sigma))}$ serves as an exact symplectic filling of the contact manifold $(\del N(\Sigma), \Ker(\alpha))$. 


\begin{theorem}\label{thm: classification}
Let $(M, \Omega)$ be a closed symplectic $6$-manifold with integral symplectic form. 
If $(M, \Omega)$ admits a symplectic hyperplane section symplectomorphic to $(\CP^1 \times \CP^1, \omega_{\FS} \oplus \omega_{\FS})$, 
then $M$ is diffeomorphic to the complex $3$-dimensional quadric 
\[
	Q^3=\{(z_0: \ldots: z_4) \in \CP^4 \mid z_0^2+\cdots+z_4^2=0\}.
\] 
\end{theorem}

The proof of the theorem relies on a classification theorem of certain $6$-manifolds.

\begin{theorem}[{Wall \cite{Wall} and Jupp \cite{Jupp}}]\label{thm: Wall and Jupp}
The following topological invariants determine the diffeomorphism types of smooth, closed, oriented, simply connected $6$-manifolds $M$ with torsion-free homology: 
\begin{itemize}
\item the two cohomology groups $H^2(M;\Z)$ and $H^3(M;\Z)$; 
\item the triple cup product $\mu \colon H^2(M; \Z) \otimes  H^2(M; \Z) \otimes  H^2(M; \Z) \rightarrow \Z$, $\mu(a,b,c)=(a\smile b \smile c)[M]$; 
\item the second Stiefel--Whitney class $w_2(M) \in H^2(M;\Z_2) \cong H^2(M;\Z) \otimes \Z_2$;  
\item the first Pontryagin class $p_1(M) \in H^4(M;\Z) \cong \Hom (H^2(M; \Z) , \Z)$. 
\end{itemize}
\end{theorem}

\begin{proof}[Proof of Theorem \ref{thm: classification}]
In what follows, we show the invariants given in Theorem \ref{thm: Wall and Jupp} are completely determined by the information of the normal bundle of $\Sigma$ in $M$, which is isomorphic to the one of $Q^2 \coloneqq \CP^1 \times \CP^1$ in $Q^3$. 
Notice that the symplectic manifold $(M,\Omega)$ can be the quadric $Q^3$ with the symplectic form induced by the Fubini--Study form on $\CP^4$ since $(\CP^1 \times \CP^1, \omega_{\FS} \oplus \omega_{\FS})$ appears as a hyperplane section in the sense of complex geometry.

Recall that the Boothby--Wang bundle associated with 
$(\CP^1 \times \CP^1,\, \omega_{\FS} \oplus \omega_{\FS})$ is contactomorphic to 
$(ST^*S^3,\xi_{\can})$. 
Hence, removing a symplectic tubular neighborhood $N(\Sigma)$ of $\Sigma$, 
we obtain an exact symplectic filling $W$ of $(ST^*S^3,\xi_{\can})$, 
which is diffeomorphic to $DT^*S^3$ by Theorem~\ref{thm: uniqueness}. 
Since $M = W \cup N(Q^2)$ and the spaces $W$, $N(Q^2)$, and 
$W \cap N(Q^2) = ST^*S^3$ are simply connected, it follows that $M$ is simply connected as well. 
To determine the homology and cohomology of $M$, 
note that $H_k(M;\Z)$ and $H^k(M;\Z)$ are straightforward to compute for all $k \neq 3$ 
using the exact sequences of the pairs $(M,W)$ and $(M,N(Q^2))$. 
For the case $k=3$, consider the homology exact sequence of $(M,N(Q^2))$: 
\[
H_3(N(Q^2);\Z) \to H_3(M;\Z) \to H_3(M,N(Q^2);\Z) \to H_2(N(Q^2);\Z) \to H_2(M;\Z),
\]
which is equivalently written as 
\[
0 \;\to\; H_3(M;\Z) \;\to\; \Z \;\to\; \Z \oplus \Z \;\to\; \Z.
\]
The group $H_2(N(Q^2);\Z)$ is generated by the two ruling curves of 
$Q^2 \cong \CP^1 \times \CP^1$, while $H_2(M;\Z)$ is generated by one such curve. 
Thus the map $H_2(N(Q^2);\Z) \to H_2(M;\Z)$ has kernel isomorphic to $\Z$, which is free, 
and hence $H_3(M,N(Q^2);\Z) \to H_2(N(Q^2);\Z)$ is injective. 
Therefore $H_3(M;\Z) \cong H^3(M;\Z) = 0$. 
We conclude that 
\[
H_k(M;\Z) \cong H_k(Q^3;\Z) 
\quad\text{and}\quad  
H^k(M;\Z) \cong H^k(Q^3;\Z) 
\quad \text{for all } k.
\]

In order to see the rest of the invariants in Theorem \ref{thm: Wall and Jupp}, we observe generators of $H_2(M;\Z)$, $H_4(M;\Z)$ and $H^2(M;\Z)$. 
In view of the above argument, the image of $[\CP^1] \in H_2(\CP^1; \Z)$ under the map induced by the inclusion $j \colon \CP^1 \hookrightarrow \CP^1 \times \{ \mathrm{pt}\} \xhookrightarrow{i} M$ gives rise to a generator of $H_2(M;\Z) \cong \Z$. 
As for the cohomology group $H^2(M;\Z)$, since 
$$
	[\Omega](j_*[\CP^1])=[\omega_{\FS}]([\CP^1])=1,
$$ 
the cohomology class $[\Omega]$ is a primitive element and hence a generator of the group. 
Similarly, the intersection 
$$
	[\Sigma] \cdot j_*[\CP^1]=1
$$
shows that $[\Sigma]=i_*[\CP^1 \times \CP^1]$ generates $H_4(M;\Z)$. 
Notice that these generators are determined by the information of the neighborhood $N(\Sigma)$ of $\Sigma$. 

Now the triple tensor product $[\Omega]^{\otimes 3}$ generates $H^2(M;\Z)^{\otimes 3}$, and the fact that $\PD([\Omega])=[\Sigma]$ gives 
\begin{align*}
	\mu([\Sigma]^{\otimes 3}) & =([\Omega]\smile [\Omega] \smile [\Omega])([M]) = ([\Omega] \smile [\Omega])([\Sigma]) \\
	 & = ([\omega_{\FS} \oplus \omega_{\FS}] \smile [\omega_{\FS} \oplus \omega_{\FS}])([\CP^1 \times \CP^1]). 
\end{align*}
As complex vector bundles, $i^*TM$ and $j^*TM$ split into 
$$
	i^*TM \cong T(\CP^1 \times \CP^1) \oplus \O_{\CP^1 \times \CP^1}(1)
\quad
{\mathrm{and}} 
\quad 
	j^*TM \cong \O_{\CP^1}\oplus \O_{\CP^1}(1) \oplus \O_{\CP^1}(2),
$$
respectively, which is evidently independent of $M \setminus N(\Sigma)$. 
According to the naturality of characteristic classes, $w_2(M)$ and $p_1(M)$ are completely determined by those pull-back bundles, which completes the proof. 
\end{proof}

\begin{remark}
To the best of the authors' knowledge, it is a difficult problem to determine whether a closed symplectic $6$-manifold $(M, \Omega)$ satisfying the assumption of Theorem~\ref{thm: classification} is symplectomorphic to the $3$-dimensional quadric $Q^3$ equipped with the standard K\"ahler form.  
Nevertheless, a result of Li and Ruan \cite{Li-Ruan} reveals a symplectic feature of $(M,\Omega)$: 
by~\cite[Corollary 4.6]{Li-Ruan}, $(M,\Omega)$ is \emph{symplectically uniruled}, as defined in terms of Gromov--Witten invariants.
\end{remark}

\begin{remark}\label{rem: CP^2}
A similar result to Theorem \ref{thm: classification} holds for $(\CP^2, \omega_{\FS})$: if a closed symplectic symplectic $6$-manifold $(M, \Omega)$ with $[\Omega] \in H^2(M;\Z)$ contains $(\CP^2, \omega_{\FS})$ as a symplectic hyperplane section, then $M$ is diffeomorphic to $\CP^3$. 
To see this, by using the result of Eliashberg--Floer--McDuff \cite{McDContact} mentioned in Remark \ref{rmk: sphere case}, one can check that such a manifold $M$ is simply connected and has the homology of $\CP^3$. 
Then, the rest of the proof goes along the same line as the one for Theorem \ref{thm: classification}.
\end{remark}

\subsection*{Acknowledgements}
The authors thank Otto van Koert and Kai Zehmisch for valuable comments, and in particular Zhengyi Zhou for pointing out a gap in the first version of this paper.
The first author was supported by the National Research Foundation of Korea (NRF) grant funded by the Korea government (MSIT) (No. NRF-2021R1F1A1060118, RS-2025-23524132).
The second author was supported by Japan Society of Promotion of Science KAKENHI Grant numbers JP20K22306, JP22K13913, JP24H00182. 

\bibliographystyle{alpha} 
\bibliography{ref}

\end{document}